\documentclass{IEEEtran}
\usepackage{cite}
\usepackage{amsmath,amssymb,amsfonts}
\usepackage{graphicx}
\usepackage{textcomp}
\usepackage{amsthm}
\usepackage{xcolor}
\usepackage{mathtools}
\usepackage{siunitx}

\usepackage[algo2e,ruled,vlined,linesnumbered,norelsize]{algorithm2e}
\makeatletter
\newcommand{\removelatexerror}{\let\@latex@error\@gobble}
\makeatother

\theoremstyle{definition}
\newtheorem{definition}{Definition}
\newtheorem{theorem}{Theorem}
\newtheorem{lemma}{Lemma}

\theoremstyle{remark}
\newtheorem{remark}{Remark}

\def\iter{i}
\def\maxiter{\iter_{\mathit{max}}}
\def\lsiter{\iter^\prime}
\def\maxlsiter{\maxiter^\prime}
\def\algtol{\mathit{tolerance}}
\def\infeasflag{\mathit{LineSearch}}
\def\old{\mathrm{old}}
\def\algfalse{\mathit{false}}
\def\algtrue{\mathit{true}}

\begin{document}
\title{Robust adaptive NMPC using ellipsoidal tubes}
\author{\begin{tabular}{@{}c@{}}
Johannes Buerger\\BMW Group, Munich, Germany\\
\texttt{johannes.buerger@bmw.de}
\end{tabular}
\qquad\qquad
\begin{tabular}{@{}c@{}}
Mark Cannon\\
University of Oxford, UK\\
\texttt{mark.cannon@eng.ox.ac.uk}
\end{tabular}
\vspace{-1.5\baselineskip}}


\maketitle

\begin{abstract}
We propose a computationally efficient nonlinear Model Predictive Control (NMPC) algorithm for safe, learning-based control. The system model is  represented as an affine combination of basis functions with unknown parameters, and is subject to additive set-bounded disturbances.
Our algorithm employs successive linearization around nominal predicted trajectories and accounts for uncertainties in predicted states due to linearization, model errors, and disturbances using ellipsoidal sets.
The ellipsoidal tube-based approach ensures that constraints on control inputs and system states are satisfied. Robustness to uncertainty is ensured using bounds on linearization errors and a backtracking line search.
We show that the ellipsoidal embedding of model uncertainty scales favourably with system dimensions in numerical simulations.
The algorithm incorporates set membership parameter estimation, and provides guarantees of recursive feasibility and input-to-state practical stability.
\end{abstract}

\begin{IEEEkeywords}
nonlinear model predictive control, adaptive control, tube model predictive control, convex optimization
\end{IEEEkeywords}

\section{Introduction}\label{introduction}
Model Predictive Control (MPC) is an optimization-based control strategy in which a finite-horizon optimal control problem is solved at each discrete time step using the current estimate of the model state, and the first control input is applied~\cite{Kou16}.
Prediction models may be constructed using diverse physical modelling and system identification methods, but residual model uncertainty is unavoidable, and this motivates learning-based and adaptive MPC strategies~\cite{Ade09,Lor19,Hew20}. In particular, robust adaptive MPC can ensure constraint satisfaction and stability while allowing online model identification.

Tube MPC achieves robust constraint satisfaction and stability by constructing tubes to bound uncertain model trajectories, and the tube geometry largely determines the trade-off between tractability and conservatism.
For linear systems, robust tube MPC is well-established, including adaptive variants with online parameter set updates~\cite{Lor19,Lu21}.  
However, to maintain tractable online optimization and receding-horizon feasibility and stability guarantees, existing approaches for nonlinear systems (e.g.~\cite{Ade09,Koh21,Sasfi23,leeman23,Bue24}) require restrictive offline controller design or lead to nonconvex online problems.
In particular~\cite{Ade09,Koh21,Sasfi23} rely on solving  a nonconvex problem online, while~\cite{leeman23} does not guarantee recursively feasible receding-horizon implementation.
Moreover~\cite{Koh21,Sasfi23} rely on computationally intensive offline design of control laws using contraction metrics.

This note describes a robust adaptive MPC algorithm   
using sequential convex approximation and ellipsoidal tubes (e.g.~\cite{kurzhanski02}) to account for model uncertainty. 
Similar to the polytopic tube MPC algorithm of~\cite{Bue24}, we assume knowledge of a locally stabilizing controller within a terminal set (this is generally easier to ensure than global incremental stability~\cite{Koh21,Sasfi23}), and solve online a sequence of convex problems with recursive feasibility and convergence guarantees. In common with other sequential convex programming algorithms for smooth problems~\cite{messerer21}, the rate of convergence is at least linear. 
The use of ellipsoidal rather than the polytopic tubes of~\cite{Bue24}  improves scalability with problem size and provides a favourable balance between computation and performance.
Here we extend the theory of~\cite{Can11} to systems with additive and parametric uncertainties, deriving a robust adaptive nonlinear MPC strategy that employs set membership parameter estimation~\cite{Lor19,Lu21} as a learning method. The resulting algorithm is formulated as a second-order cone program.

The MPC optimization is convexified using linearization, and to ensure recursive feasibility and stability, the perturbations around state and control linearization points are confined to polytopic sets and approximation errors are bounded using ellipsoidal tubes containing predicted state trajectories. 
By integrating a line search procedure into the online algorithm, this approach provides recursive feasibility and performance guarantees even if only one optimization iteration is performed at each online time step.
This note extends the preliminary publication~\cite{Buerger25} by demonstrating input-to-state practical stability, providing additional implementation details, and by extending the numerical analysis to illustrate how computation and performance scale with problem size in comparison to polytopic tube robust nonlinear MPC.

The remainder of this note is organised as follows. Section~\ref{sec:problem_statement} introduces the control problem, 
Sections~\ref{sec:error_bounds}, \ref{sec:tube_conditions} and \ref{sec:opt} discuss approximation error bounds, ellipsoidal tube MPC and online optimization, respectively. Section~\ref{sec:termset} derives terminal conditions, Section~\ref{sec:stability} analyses feasibility and stability, and Section~\ref{sec:param_est} discusses parameter estimation.
Section~\ref{sec:simex} investigates scaling of computation and performance with system dimensions, and conclusions are presented in Section~\ref{sec:conclusion}.

\textit{Notation:} $\mathbb{N}_{\geq 0}$ is the set of non-negative integers, $\mathbb{N}_{[p,q]} = \{n\in\mathbb{N} : {p \leq n \leq q}\}$, and $\mathbb{N}_q =  \mathbb{N}_{[1,q]}$.
A vector with all elements equal to 1 is denoted $\underline{1}$ and the $n\times n$ identity matrix is $I_n$.
The $i$th row of a matrix $A$ is $[A]_i$.
For a matrix $A$, the inequality $A \geq 0$ applies elementwise, and $A \succeq 0$ (or $A \succ 0$) indicates that $A$ is positive semidefinite (positive definite).  
The Euclidean norm is $\|x\|$ and $\|x\|_Q = (x^\top Q x)^{1/2}$ for $Q\succeq 0$.  

\section{Optimal control problem}\label{sec:problem_statement}

We consider a discrete time system with state $x$, control input $u$, unknown parameter vector $\theta$, and additive disturbance~$w$
\begin{equation}\label{eq:system}
x_{t+1} = f(x_t,u_t,\theta) + w_t ,
\end{equation}
where $(x_t,u_t) \in \mathcal{X}\times\mathcal{U}$, $\theta \in \Theta_0$, $w_t\in \mathcal{W}$, and $t$ is the discrete time index.
We assume $\mathcal{X}$, $\mathcal{U}$, $\Theta_0$, $\mathcal{W}$ are known polytopic sets:
$\mathcal{U} = \{u \in \mathbb{R}^{n_u} : G u \leq 1\}$,
$\mathcal{X} = \{x \in \mathbb{R}^{n_x}  : E x \leq 1\}$,
$\Theta_0 = \{\theta \in \mathbb{R}^{n_{\theta}} : H_{\Theta} \theta \leq h_0\}=\mathrm{co}\{\theta^{(q)},\, q \in\mathbb{N}_{\nu_{\theta}}\}$, and $\mathcal{W}=\mathrm{co}\{w^{(r)}, \, r \in\mathbb{N}_{\nu_{w}}\}$.

The function $f$ is assumed to be an affine combination of known basis functions $f_i(x,u)$, $i \in\mathbb{N}_{n_{\theta}}$
\begin{equation}\label{eq:theta_expansion}
f\bigl(x,u,(\theta_1,\ldots,\theta_{n_\theta})\bigr) = f_0(x,u) + \sum_{i=1}^{n_{\theta}} {\theta_i f_i(x,u)} ,
\end{equation}
where,  for each $i$, $f_i$ is differentiable and Lipschitz continuous on $\mathcal{X}\times\mathcal{U}$ with $f_i(0,0) = 0$.
The proposed approach can be used with any parameter estimator that provides a compact polytopic set $\Theta_t$ satisfying $\theta \in \Theta_t \subseteq \Theta_{t-1}$ for all $t>0$.

The control objective is to minimize an upper bound over model uncertainty of a quadratic cost with weights $Q,R\succ0$:
\begin{equation} \label{eq:cost}
\sum_{t=0}^{\infty} (\left\lVert x_t \right\rVert^2_Q +  \left\lVert u_t \right\rVert^2_R).
\end{equation}
To mitigate the effects of model uncertainty on predicted future state and control sequences, the control decision variables are perturbations $\{v_t, v_{t+1} ,\ldots\}$ applied to a feedback law
\[
u_t = K x_t + v_t ,
\]
where the feedback gain $K$ is robustly stabilizing, at least locally around $x=0$, in the sense defined in Section~\ref{sec:termset}.
To simplify notation, for all $(x,v)\in\mathcal{X}\times\mathcal{U}$ and $\theta\in\Theta_0$, let
\[
f_{K}(x,v,\theta) = f(x,Kx+v,\theta)
\]
and let $f_{K,i}(x,v) = f_i(x,Kx+v)$ for $i=0,\ldots,n_{\theta}$.

\section{Linearization error bounds}\label{sec:error_bounds}
We consider the Taylor expansion of the model in~\eqref{eq:system} around a nominal trajectory $\mathbf{x}^0 = \{x^0_0,\ldots, x^0_{N}\}$, which is defined for given
$\mathbf{v}^0 = \{v^0_0, \ldots, v^0_{N-1}\}$ 
and $\theta^0\in\Theta$ 
by
\begin{equation}\label{eq:nominal_system}
x^0_{k+1} = f_K(x_k^0, v^0_k, \theta^0), \ k = 0,\ldots, N-1.
\end{equation}
Defining state and control perturbations $s_k$ and $v_k$, where $x_k = x_k^0 + s_k$ and $u_k = Kx_k + v^0_k + v_k$, we have
\begin{equation}\label{eq:taylor}
x^0_{k+1}+s_{k+1} =  f_K(x^0_k,v^0_k,\theta^0) + \delta^0_k + \Phi_k s_k + B_k v_k + \delta^1_k + w_k
\end{equation}
where $\Phi_k = \nabla_x f_K (x^0_k,v^0_k,\theta^0)$ and $B_k = \nabla_v f_K (x^0_k,v^0_k,\theta^0)$ denote the Jacobian matrices of $f_K$ with respect to $x$ and~$v$.

The perturbations on the state and control input are constrained to satisfy $s_k\in \mathcal{S}$, $v_k\in \mathcal{V}$ for all $k\in \mathbb{N}_{[0,N]}$, where $\mathcal{S}$, $\mathcal{V}$ are bounded polytopic sets containing the origin. The perturbation $s_k$ contains a nominal component and an uncertain component (caused by linearization errors, parameter errors and disturbances occurring prior to the $k$th time step). In this context we distinguish the zero-order error term $\delta^0_k$ and the first-order error term $\delta^1_k$
as follows.

\begin{definition}\label{def:delta} The error terms $\delta^0_k$ and $\delta^1_k$ in \eqref{eq:taylor} are defined by
\begin{align*}
\delta^0_k &= f_K(x_k^0, v^0_k, \theta) - f_K(x_k^0, v^0_k, \theta^0)
\\
\delta^1_k &= f_K(x_k^0 + s_k, v^0_k + v_k, \theta) - f_K(x_k^0, v^0_k, \theta) - \Phi_k s_k \!-\! B_k v_k
\end{align*}
\end{definition}

Bounds on $\delta^0_k$ are derived from the affine dependence on $\theta$:
\[
\delta^0_k = f_K(x^0_k, v^0_k,{\theta}) -f_K(x^0_k, v^0_k,\theta^0) = f_K (x^0_k, v^0_k, \theta - \theta^0) ,
\]
which implies a polytopic additive disturbance bound: 
\begin{equation}\label{eq:W0_def}
\delta^0_k \in \mathcal{W}^0_k = \mathcal{W}^0(x^0_k,v^0_k,\theta^0,\Theta) = \mathrm{co}
\{\delta^{0\,(q)}_k, \, q \in \mathbb{N}_{\nu_\theta}\}.
\end{equation}
A tight bound can be obtained by elementwise maximization over $\theta \in \Theta$. The set $\mathcal{W}_k^0$ can be recomputed online using the current parameter set estimate and nominal trajectory.

Bounds on $\delta^1_k$ can similarly be derived using Definition~\ref{def:delta}. The mean value theorem implies, for some $(s,v)\in \mathcal{S}\times \mathcal{V}$,
\begin{align}
\delta^1_k &= \bigl(\nabla_x f_K(x^0_k+s, v^0_k+v, {\theta}) - \Phi_k\bigr) s_k
\nonumber \\
&\quad + \bigl(\nabla_v f_K (x^0_k+s, v^0_k+v, {\theta})- B_k \bigr) v_k . 
\label{eq:mvt_delta1}
\end{align}
A corresponding polytopic uncertainty set can be defined
\begin{align}
\delta^1_k \in \mathcal{W}^1_k 
&= \mathcal{W}^1  (x^0_k,v^0_k,\theta^0, \mathcal{S}, \mathcal{V}, \Theta) \nonumber
\\
&\quad = \mathrm{co} \{ C^{(j)}_ks_k + D^{(j)}_k v_k, \, j \in \mathbb{N}_{\nu_{1}} \}
\label{eq:W1_def}
\end{align}
where $\{C^{(j)}_k,D^{(j)}_k,\, j \in \mathbb{N}_{\nu_{1}}\} $ are determined, for example, from componentwise bounds on the Jacobian matrices in~(\ref{eq:mvt_delta1}) for $(s,v,\theta)\in \mathcal{S}\times\mathcal{V}\times \Theta$.
The computation required is example-dependent. In some cases (such as the examples in Section~\ref{sec:simex}) the bounds are attained at the extreme points of $\mathcal{S}, \mathcal{V}$,
allowing $\mathcal{W}_k^1$ to be computed online using the current nominal trajectory and updated parameter set estimate. More generally $\mathcal{W}_k^1$ may be determined offline as a global bound over $\mathcal{X}$, $\mathcal{U}$, and $\Theta_0$.

\section{Ellipsoidal tube MPC}\label{sec:tube_conditions}
From~(\ref{eq:taylor}), the state perturbation $s_k$ evolves according to
\begin{equation}
s_{k+1} = \Phi_k s_k + B_k v_k + w_k + \delta^0_k + \delta^1_k
\end{equation}
where $w\in \mathcal{W}$, $\delta^0_k\in\mathcal{W}^0_k$ and $\delta^1_k\in\mathcal{W}^1_k$.
Using a conventional tube MPC strategy (e.g.~\cite{Kou16}) we split $s_k$ into nominal and uncertain components $z_k$ and $e_k$ satisfying
\begin{align}
s_k &= z_k + e_k \label{eq:decomp} \\
z_{k+1} &= \Phi_k z_k + B_kv_k \label{eq:z_dynamics}\\
e_{k+1} &= \Phi_k e_k + w_k + \delta^0_k + \delta^1_k
\label{eq:e_dynamics}
\end{align} 
for $k = 0,\ldots,N-1$, and we construct a sequence of ellipsoidal sets bounding the uncertain component $e_k$,
\begin{equation}\label{eq:tube_cond1}
e_k \in \mathcal{E}(V, \beta_k^2), 
\quad k = 0,\ldots,N 
\end{equation}
where $\mathcal{E}(V, \beta^2)= \{e: e^\top V e\leq \beta^2\}$.
The tube geometry is determined by the matrix $V\succ 0$, which is computed offline as discussed in Section~\ref{sec:termset}. The scalings $\beta_k$ are optimized online.

To enforce~(\ref{eq:tube_cond1}), we require, for all $k=0,\ldots,N-1$
\begin{align*}
& \mathcal{E}(V, \beta_{k+1}^2) \ni \Phi_k e + w + \delta^0 + \delta^1 , \\ 
& \qquad \forall w \in \mathcal{W}, 
\ \forall \delta^0 \in \mathcal{W}^0_k, 
\ \forall \delta^1 \in \mathcal{W}^1_k, 
\ \forall e\in \mathcal{E}(V, \beta_k^2) .
\end{align*}
A sufficient condition for this is given by, for all $j\in\mathbb{N}_{\nu_{1}}$, $q\in\mathbb{N}_{\nu_\theta}$, $r\in\mathbb{N}_{\nu_w}$, and all $e\in\mathcal{E}(V,\beta_k^2)$,
\begin{equation}\label{eq:tube_mem_cond_beta} 
\beta_{k+1} \geq 
\lVert C^{(j)}_k z_k + D^{(j)}_k v_k + \delta^{0\,(q)}_k \rVert_V + \lVert (\Phi_k + C^{(j)}_k)e + w^{(r)}\rVert_V .
\end{equation}
We enforce~(\ref{eq:tube_mem_cond_beta}) using the following observation.

\begin{lemma}\label{lem:beta_mode1_dynamics}
Condition (\ref{eq:tube_mem_cond_beta}) holds for all $e\in\mathcal{E}(V,\beta_k^2)$ if
\begin{equation}
\beta_{k+1} \geq (\lambda_k \beta_k^2 + \sigma^2)^{\frac{1}{2}} + \lVert C^{(j)}_k z_k + D^{(j)}_k v_k + \delta^{0\,(q)}_k \rVert_V
\label{eq:beta_mode1_dynamics}
\end{equation}
for all $j\in\mathbb{N}_{\nu_{1}}$, $q\in\mathbb{N}_{\nu_\theta}$, with $\lambda_k$ defined by
\begin{equation}
\lambda_k = \max_{j\in\mathbb{N}_{\nu_1},\, r\in\mathbb{N}_{\nu_w}} \lVert
(\Phi_k + C_k^{(j)})V^{-\frac{1}{2}} \rVert_{\Psi^{(r)}}^2
\label{eq:lambda_mode1_def}
\end{equation}
where $\Psi^{(r)} = (V^{-1} - w^{(r)} w^{(r)\, \top} \sigma^{-2})^{-1}$ and where $\sigma$ is a constant whose design is discussed in Section~\ref{sec:termset}.
\end{lemma}

\begin{proof}
Let $\mu^{(j,q)} = \beta_{k+1} - \lVert C^{(j)}_k z_k + D^{(j)}_k v_k + \delta^{0\,(q)}_k \rVert_V$, then~(\ref{eq:tube_mem_cond_beta}) holds if and only $\lambda_k^{(j,r)} \geq 0$ exists satisfying%
\def\strut{\rule{0pt}{8pt}}%
\begin{equation}\label{eq:lambda_mode1_cond}
\left[\begin{smallmatrix}
\lambda_k^{(j,r)} V & ~~~0~~~ & (\Phi_k + C^{(j)}_k)^\top 
\\
\star\strut & \sigma^2 & {w^{(r)}}^\top
\\
\star\strut & \star & V^{-1}
\end{smallmatrix}\right] \succeq 0 ,
\end{equation}
and ${\mu^{(j,q)}}^2 \geq  \lambda_k^{(j,r)}\beta_{k}^2 + \sigma^2$.
By Schur complements,~(\ref{eq:lambda_mode1_cond}) holds iff $\lambda_k^{(j,r)} V \succeq (\Phi_k + C^{(j)}_k)^\top 
\Psi^{(r)}(\Phi_k + C^{(j)}_k)$.
Choosing $\lambda_k$ as the smallest scalar satisfying $\lambda_k \geq \lambda_k^{(j,r)}$ therefore yields the sufficient conditions~{(\ref{eq:beta_mode1_dynamics}) and
(\ref{eq:lambda_mode1_def})}.
\end{proof} 

We determine the MPC law by minimizing a bound on the maximum over the ellipsoidal tube of the predicted cost:
\begin{multline}\label{eq:maxcost}
\sum_{k=0}^{N-1} \max_{x_k \in x^0_k + z_k + \mathcal{E}(V,\beta_k^2)}
  ( \lVert x_k \rVert^2_Q +  \lVert Kx_k +  v^0_k+ v_k \rVert^2_R) \\
  + \hat{l}^2(z_N,\beta_N,x^0_N)
\end{multline}
where $\hat{l}$ is a terminal cost whose design is discussed in Section~\ref{sec:termset}.
This minimization is performed over the variables $\mathbf{v}=\{v_0,\ldots,v_{N-1}\}$, $\boldsymbol{\beta} = \{\beta_0,\ldots,\beta_N\}$, $\mathbf{z} =\{z_0,\ldots,z_N\}$, 
subject to input, state, and terminal constraints, and constraints to enforce a sufficient cost decrease, as described in Section~\ref{sec:opt}.

\section{Online MPC optimization}\label{sec:opt}

At discrete time $t$, iteration $i$ of the MPC optimization solves the following Second Order Cone Program (SOCP) given the current plant state $x^p_t$.
\begin{align} \label{opt:mpc}%
& (\mathbf{v}^\star,\boldsymbol{\beta}^\star,\mathbf{z}^\star,\mathbf{l}^\star, \bar{J}_t^{(i)}) = 
\arg\min_{\mathbf{v},\boldsymbol{\beta},\mathbf{z},\mathbf{l},\bar{J}_t^{(i)}}\bar{J}_t^{(i)}
\\ \nonumber
& \text{subject to, for $k=0,...,N-1$, and all $j\in\mathbb{N}_{\nu_{1}}$, $q\in\mathbb{N}_{\nu_\theta}$, 
}\\  \nonumber
& \quad \begin{aligned}
\bar{J}_t^{(i)} & \geq \lVert [ l_0 \ l_1 \ \cdots \ l_{N-1} \ l_{N} ] \rVert^2 \\
z_{k+1} & = \Phi_k z_k + B_kv_k \\
\beta_{k+1} & \geq 
(\lambda_k\beta_k^2 + \sigma^2)^\frac{1}{2} + 
\lVert C^{(j)}_k z_k + D^{(j)}_k v_k + \delta^{0\,(q)}_k \rVert_V \\
l_{k} & \geq \bigl( \lVert x^0_k +z_k\rVert_Q^2 
+ \lVert K(x^0_k + z_k) + v^0_k + v_k\rVert_R^2  \bigr)^{\frac{1}{2}}
\\
& \quad + \beta_k \lVert{V}^{-\frac{1}{2}}\rVert_{Q+K^\top R K}
\\
\mathcal{U} & \supset K\bigl(x_k^0+z_k+\mathcal{E}(V,\beta_k^2)\bigr) + v_k^0 + v_k \\
\mathcal{X} & \supset x_k^0+z_k+\mathcal{E}(V,\beta_k^2) \\
\mathcal{V} & \ni v_k^0 + v_k \\
\mathcal{S} & \supset z_k+\mathcal{E}(V,\beta_k^2) \\
\end{aligned} \\ \nonumber
& \text{and initial and terminal conditions} \\ \nonumber
& \quad \begin{aligned}
\beta_0 & \geq \lVert x^0_0 + z_0 - x^p_t\rVert_V \\
\Omega(x^0_N) & \ni (\|z_N\|_V, \beta_N) \\
l_N &\geq \hat{l}(z_N, \beta_N, x^0_N )
\end{aligned} \\ \nonumber
& \text{and, for iteration $i=1$,} \\ \nonumber
& \quad \bar{J}^{(i)}_t \leq \bar{J}^{(\mathit{final})}_{t-1} - \bigl(\lVert x_{t-1} \rVert^2_Q +  \lVert u_{t-1} \rVert^2_R \bigr) + \hat{\sigma}^2 \\ \nonumber
&\text{and, for iterations $i>1$,} \\  \nonumber
& \quad \bar{J}^{(i)}_t \leq \bar{J}^{(i-1)}_t .
\end{align}
Here, the constraints on $z_{k+1}$ and $\beta_{k+1}$ ensure that $x_k^0+z_k+\mathcal{E}(V,\beta_k^2)$ bounds the predicted state trajectories, while $l_k^2$ bounds the $k$th term in the cost~(\ref{eq:maxcost}) via the triangle inequality.
Constraints of the form $\{x : H x \leq h\} \supset z + \mathcal{E}(V,\beta)$ enforcing the control, state and tube bounds are imposed via $[H]_r z + \beta \|V^{-\frac{1}{2}}[H]_r^\top\| \leq [h]_r$ for each row $r$ of $H$.
The design of the terminal set $\Omega(x^0_N)$, terminal cost $\hat{l}$, and $\hat{\sigma}$ are discussed in Section~\ref{sec:termset}, and $\bar{J}_{t}^{(i-1)}$, $\bar{J}_{t-1}^{(\mathit{final})}$ denote the optimal objective respectively at iteration $i-1$ and at the final iteration at time $t-1$.
The MPC strategy is summarised in Algorithm~\ref{alg:mpc}.

{\setlength{\algomargin}{0.83em}
\removelatexerror
\begin{algorithm2e}[hbt]
\DontPrintSemicolon
\SetKwInOut{Input}{Input}
\SetKwInOut{Output}{Output}
\Input{Initial perturbation sequence $\mathbf{v}^0$;
parameter estimate $\theta^0$; 
uncertainty bounds $\mathcal{W},\Theta$; 
cost weights $Q,R$;
tube parameters $\mathcal{S},\mathcal{V},V,\sigma$}
\Output{Control input $u_t$ at time steps $t = 0,1,\ldots$}
At time $t$, set $x^0_0 \gets x^p_t$, $\iter \gets 1$\;
\While{$\iter \leq \maxiter$ and $\|{\mathbf{v}^{\star}}\|\geq \algtol$}{
Compute $\mathbf{x}^0$ by simulating the nominal system 
(\ref{eq:nominal_system}) with initial state $x^0_0$ and perturbation sequence $\mathbf{v}^0$\;
Compute $\Phi_k,B_k$ in (\ref{eq:taylor}), bounds $\mathcal{W}^0_k, \mathcal{W}^1_k$, in (\ref{eq:W0_def}), (\ref{eq:W1_def}) and $\lambda_k$ in (\ref{eq:lambda_mode1_def}) for all $k\in \mathbb{N}_{[0,N-1]}$\; 
Attempt to solve Problem~{\rm(\ref{opt:mpc})} for $\mathbf{v}^{\star}$\; 
\If{Problem~{\rm(\ref{opt:mpc})} is infeasible}{
$\alpha \gets 1$, $\infeasflag \gets \algtrue$, $\lsiter \gets 1$\;
\While{$\lsiter \leq \maxlsiter$ and $\infeasflag = \algtrue$}{
$\alpha \gets \alpha/2$\;
\If{$\iter = 1$}{
$x^0_0 \gets x^0_{0,\old} + \alpha (x^0_0 -x^0_{0,\old})$\; 
}
$\mathbf{v}^0\gets \mathbf{v}^0_{\old} + \alpha (\mathbf{v}^0 -\mathbf{v}^0_{\old})$, $\lsiter \gets \lsiter + 1$\; 
Compute $\mathbf{x}^0$ using 
(\ref{eq:nominal_system}) with $x^0_0$ and $\mathbf{v}^0$\;
Compute $\Phi_k, B_k, \mathcal{W}^0_k, \mathcal{W}^1_k, \lambda_k, \forall k\in \mathbb{N}_{[0,N-1]}$\!\!\!\!\!\; 
Attempt to solve Problem~{\rm(\ref{opt:mpc})} for $\mathbf{v}^{\star}$\; 
\If{Problem~{\rm(\ref{opt:mpc})} is feasible}{
$\infeasflag \gets \algfalse$\;
}
}
\If{$\infeasflag = \algtrue$}{
$\mathbf{v}^{\star}\gets 0$, 
$\mathbf{v}^0 \gets \mathbf{v}^0_{\old}$, $i\gets \maxiter$\;
\If{$\iter = 1$}{
$\mathbf{x}^0 \gets \mathbf{x}^0_{\old}$\; 
}
}
}
Store $\mathbf{v}^0_{\old} \gets \mathbf{v}^0$\;
Update $\mathbf{v}^0 \gets \mathbf{v}^0 +\mathbf{v}^{\star}$, $\iter\gets \iter+1$\; 
}
Apply the control input $u_t= K x^p_t+ v^0_0$ and set
$\mathbf{v}^0 \gets \{v^0_1, \ldots, v^0_{N-1},0\}$
$\mathbf{v}^0_{\old} \gets \{v^0_{1,\old}, \ldots, v^0_{N-1,\old},0\}$
$\mathbf{x}^0_{\old} \gets \{x^0_{1},\ldots,x^0_{N},f_K(x_{N}^0,0,\theta^0)\}$\; 
\caption{Ellipsoidal Tube MPC}\label{alg:mpc}
\end{algorithm2e}}

The main iteration of Algorithm~\ref{alg:mpc} (lines 2-23) computes the Jacobian linear approximation~(\ref{eq:taylor}) of the model (\ref{eq:system}) about the nominal trajectory generated by $\mathbf{v}^0$ (lines 3-4), and attempts to solve Problem~(\ref{opt:mpc}) (line 5).
If the solver is capable of warm-starts, Problem~(\ref{opt:mpc}) can be initialised with ${\bf v} =0$ to leverage the solution of the previous iteration.
Although the constraints of Problem~(\ref{opt:mpc}) ensure that 
the nominal trajectory $\mathbf{x}^0$
satisfies $(x_k^0,u_k^0) \in (\mathcal{X},\mathcal{U})$ $\forall k\in \mathbb{N}_{[0,N-1]}$ and $x_{N}^0 \in \mathcal{X}_N$, the linearized dynamics computed in line 4 may not define a feasible constraint set for Problem~(\ref{opt:mpc}). We therefore perform a backtracking line search in lines~8-21,  exploiting knowledge of a perturbation $\mathbf{v}^0_{\mathrm{old}}$ and initial nominal state $x^0_{0,\mathrm{old}}$ such that Problem~(\ref{opt:mpc}) is feasible. As we show in Section~\ref{sec:stability}, this provides a guarantee of feasibility of Problem~(\ref{opt:mpc}) for all~$t$.

\section{Terminal constraints and terminal cost} \label{sec:termset}
This section discusses how to compute the parameters $V,K,\sigma,\hat{\sigma}$ in problem~\eqref{opt:mpc}
and how to construct the terminal set $\Omega$ and terminal cost $\hat{l}(z,\beta,x^0)$.
In order to design a terminal cost and constraint set providing recursive feasibility and stability guarantees, we define the predicted trajectories of the model~(\ref{eq:system}) beyond the initial $N$-step prediction horizon by setting $v_k = 0$, $v^0_k = 0$ and $u_k=K x_k$ for  $k\geq N$.
To allow robust linear control design methods, we construct a linear difference inclusion (LDI) (e.g.~\cite{Boy94}) for each basis function $f_i$ in~($\ref{eq:theta_expansion}$) in a neighbourhood of $(x,u)=(0,0)$. Considering all combinations of individual LDIs over the parameter set $\Theta_0$ yields an aggregate LDI. 
For all $\theta \in \Theta_0$, 
$x\in\hat{\mathcal{X}}$, $u\in\hat{\mathcal{U}}$,
for given polytopic sets $\hat{\mathcal{X}}$, $\hat{\mathcal{U}}$ containing the origin, let
\[
[\begin{matrix}
\nabla_x f(x, u, \theta) & \nabla_u f(x,u,\theta)\end{matrix}] 
\in \mathrm{co} \bigl\{ [\begin{matrix}\hat{A}^{(j)}  & \hat{B}^{(j)} \end{matrix}], \, j \in \mathbb{N}_{\hat{\nu}} \bigr\} ,
\]
then, for all $(x_k,u_k)\in\hat{\mathcal{X}}\times\hat{\mathcal{U}}$ we have
\begin{equation}
f(x_k,u_k,\theta) \in\mathrm{co}\{\hat{A}^{(j)} x_k + \hat{B}^{(j)} u_k, \, j \in \mathbb{N}_{\hat{\nu}} \} .
\label{eq:ldi_term}
\end{equation} 
To simplify notation, let $\bar{\mathcal{X}} = \hat{\mathcal{X}}\cap\{x: Kx \in \mathcal{\hat{U}}\}$.

\begin{remark}
The LDI~(\ref{eq:ldi_term}) can be computed using the vertices of $\Theta_0$ and bounds on the Jacobians of the basis functions $f_i$ in (\ref{eq:theta_expansion}), analogously to the bounds on $\delta_k^1$ in Section~\ref{sec:error_bounds}.
\end{remark}

\begin{lemma} \label{lem:cost_bound}
For all $x\in\bar{\mathcal{X}}$
the inequality
\begin{equation} \label{eq:cost_bound}
\lVert x \rVert^2_V - \lVert f(x,Kx,\theta) +w \rVert^2_V
\geq \lVert x \rVert^2_Q +  \lVert K x  \rVert^2_R - \sigma^2
\end{equation}
holds for all $\theta\in \Theta_0$ and all $w\in\mathcal{W}$ for positive definite $V$ and some scalar $\sigma$ if the following Linear Matrix Inequality (LMI) holds for all $j\in \mathbb{N}_{\hat{\nu}}$ and $ r \in \mathbb{N}_{\nu_w}$:%
\def\strut{\rule{0pt}{8pt}}%
\begin{equation} \label{eq:cost_lmi}
\left[\begin{smallmatrix}
S & ~~~0~~~ & (\hat{A}^{(j)}S + \hat{B}^{(j)}Y)^\top & S & ~~Y^\top \\
\star & \tau & {w^{(r)}}^\top & 0 & 0 \\
\star\strut & \star & S & 0 & 0 \\
\star\strut & \star & \star & Q^{-1} & 0 \\
\star\strut & \star & \star & \star & ~~R^{-1}
\end{smallmatrix}\right] \succeq 0
\end{equation} with $V=S^{-1}$, $\sigma^2=\tau$, and $K =YV$.
\end{lemma}

\begin{proof}
This follows by substituting (\ref{eq:ldi_term}) into (\ref{eq:cost_bound})
and considering Schur complements of~(\ref{eq:cost_lmi}).
\end{proof}

To compute $V$, $\sigma$ and $K$ we minimize $\tau$ subject to (\ref{eq:cost_lmi}) by solving a semidefinite program. This approach is justified since (\ref{eq:cost_bound}) implies that $\sigma^2$ bounds the time-average value of the stage cost in (\ref{eq:cost}) as $t\to\infty$ under the control law $u_t=Kx_t$.

To determine the terminal constraint set for the tube parameters in Problem~(\ref{opt:mpc}) we first introduce the notation $\hat{Q} = Q+K^\top R K$, $\hat{\Phi}^{(j)} = \hat{A}^{(j)}+\hat{B}^{(j)}K$, $j\in\mathbb{N}_{\hat{\nu}}$, and define
\begin{align}
\hat{\lambda} &= 1 - \sigma_{\min} ( V^{-\frac{1}{2}} \hat{Q} V^{-\frac{1}{2}} )
\label{eq:lambda_term}
\\
d_\Theta &= \max_{\theta^0,\theta^1 \in\Theta_0} \| \theta^0 - \theta^1 \|_1
\label{eq:Theta_diam}
\\
d_{\hat{\Phi}} &= \max_{j,k\in\mathbb{N}_{\hat{\nu}}} \|
\hat{\Phi}^{(j)} - \hat{\Phi}^{(k)}\|_V .
\label{eq:Phi_diam}
\end{align}%
Here $\hat{\lambda}\in [0,1)$ since a Schur complement of (\ref{eq:cost_lmi}) implies $V\succeq \hat{Q}$.
We further assume $f_{K,i}(x,0)$ for each $i\in\{0,\ldots,p\}$ is $L$-Lipschitz continuous with respect to $x$, for all $x\in\bar{\mathcal{X}}$. 

\begin{lemma}\label{lem:term_tube}
If $v_k=v^0_k=0$
in~(\ref{eq:taylor}) and (\ref{eq:decomp})-(\ref{eq:e_dynamics}), and if 
$x^0_k\in\bar{\mathcal{X}}$ and $x_k \in x^0_k+z_k+\mathcal{E}(V,\beta_k^2)\subseteq\bar{\mathcal{X}}$,
then $e_k \in \mathcal{E}(V,\beta_k^2)$ and
\begin{align}
\|x^0_{k+1}\|_V &\leq \hat{\lambda}^{\frac{1}{2}} \|x^0_k\|_V 
\label{eq:x0_term_bound}
\\
\|z_{k+1}\|_V &\leq \hat{\lambda}^{\frac{i}{2}} \|z_k\|_V 
\label{eq:z_term_bound}
\\
\beta_{k+1} &\geq (\hat{\lambda} \beta_{k}^2 +\sigma^2)^\frac{1}{2} + d_{\hat{\Phi}} \|z_{k}\|_V + 
d_\Theta L \|x_k^0\|_V .
\label{eq:beta_term_bound}
\end{align}
\end{lemma}

\begin{proof}
From (\ref{eq:ldi_term}) with $u_k=Kx_k$ and $x^0_k,z_k,e_k\in \bar{\mathcal{X}}$ we get
\begin{align*}
x^0_{k+1} &= f_K(x^0_k , 0, \theta^0)
\in\mathrm{co}\{\hat{\Phi}^{(j)} x^0_k , \, j \in \mathbb{N}_{\hat{\nu}} \} 
\\
z_{k+1} &= \Phi_k z_k 
\in\mathrm{co}\{\hat{\Phi}^{(j)} z_k , \, j \in \mathbb{N}_{\hat{\nu}} \} 
\\
e_{k+1} &= \Phi_k e_k + w_k + \delta^0_k + \delta^1_k 
\\
&\in\mathrm{co}\{\hat{\Phi}^{(j)} e_k , \, j \in \mathbb{N}_{\hat{\nu}} \} + w_k + \mathrm{co}\{ (\hat{\Phi}^{(j)} \!-\! \Phi_k) z_k \} + \delta^0_k 
\end{align*}
but (\ref{eq:cost_lmi}) ensures that, for all  $w\in\mathcal{W}$,  $j \in \mathbb{N}_{\hat{\nu}}$, and all $x\in\mathbb{R}^{n_x}$,
\begin{align*}
\|\hat{\Phi}^{(j)} x\|_V^2 
&\leq \|x\|_V^2 - \|x\|_{\hat{Q}}^2 \leq \hat{\lambda} \|x\|_V^2
\\
\|\hat{\Phi}^{(j)} x + w\|_V^2 
&\leq \|x\|_V^2 - \|x\|_{\hat{Q}}^2 + \sigma^2 \leq \hat{\lambda} \|x\|_V^2 + \sigma^2
\end{align*}
since (\ref{eq:lambda_term}) implies $V - \hat{Q} \preceq \hat{\lambda} V$. 
Furthermore, for all $j\in\mathbb{N}_{\hat{\nu}}$,
\[
\|\mathrm{co}\{ (\hat{\Phi}^{(j)} - \Phi_k) z_k \}\|_V 
\leq 
d_{\hat{\Phi}} \| z_k\|_V ,
\]
and $\delta^0_k = f_K(x_k^0,0,\theta - \theta^0)$ satisfies
\[
\|\delta^0_k\|_V \leq L \|x_k^0\|_V \|\theta \!-\! \theta^0\|_1
\leq 
d_{\Theta} L \| x^0_k\|_V .
\]
Hence $x^0_k$, $z_k$ satisfy (\ref{eq:x0_term_bound}), (\ref{eq:z_term_bound}), and $e_k$ satisfies  
\[
\|e_{k+1}\|_V \leq 
(\hat{\lambda} \|e_{k}\|_V^2 +\sigma^2)^\frac{1}{2}
+ d_{\hat{\Phi}} \|z_{k}\|_V + 
d_\Theta L \|x_k^0\|_V ,
\]
from which the bound~(\ref{eq:beta_term_bound}) follows.
\end{proof}

The terminal set $\Omega(x^0_N)$ is constructed so that 
$Hx_k \leq h$ for all $k\geq N$ whenever $(\|z_N\|_V,\beta_N) \in\Omega(x^0_N)$, where
\[
\{x : Hx \leq h\}
 = \mathcal{X} \cap \hat{\mathcal{X}} \cap \{x : Kx \in \mathcal{U} \cap \hat{\mathcal{U}}\} 
\]
is the aggregate constraint set.
To ensure this we impose the condition $\|x^0_k + z_k\|_V + \beta_k \leq \hat{\rho}$ for all $k\geq N$, where 
\begin{equation}\label{eq:rho_def}
\hat{\rho} = \min_i \bigl\{ [h]_i/\|[H]_i^\top\|_{V^{-1}}\bigr\} ,
\end{equation}
by defining the terminal constraint set for $(\|z_N\|_V, \beta_N)$ as%
\begin{align}
& \Omega(x^0_N) = \bigl\{
(r,\beta_N) : 
\beta_N \leq \hat{\rho} - (r + \|x^0_N\|_V)
\nonumber\\
& \quad \text{and } \exists \beta_{k} \text{ satisfying, for }
k=N+1,\ldots, N+\hat{N}, 
\nonumber\\
& \quad \beta_{k} \geq 
(\hat{\lambda}\beta_{k-1}^2 + \sigma^2)^{\frac{1}{2}}
+ \hat{\lambda}^{\frac{(k-N-1)}{2}} (r d_{\hat{\Phi}} +  d_\Theta L \|x^0_N\|_{V}) ,
\nonumber\\
& \quad \beta_{k}
\leq \hat{\rho} - \hat{\lambda}^\frac{k-N}{2}(r + \|x^0_N\|_V) \bigr\}
\label{eq:term_constraint_set}
\end{align}
with $\hat{N}$ chosen large enough to satisfy 
\begin{align}
\max_{(r,\beta_N)\in\Omega(x^0_N)} & \bigl\{
(\hat{\lambda}\beta_{N+\hat{N}\!}^2 + \sigma^2)^{\frac{1}{2}\!} 
+ \hat{\lambda}^\frac{\hat{N}}{2}\! (r d_{\hat{\Phi}} +  d_\Theta L \|x^0_N\|_V) 
\nonumber \\
&\quad + \hat{\lambda}^\frac{\hat{N}+1}{2} (r +  \|x_k^0\|_V \bigr\} 
\leq \hat{\rho}
.
\label{eq:Nc_def}
\end{align}

\begin{lemma}\label{lem:termset}
If $(\|z_N\|_V,\beta_N)\in\Omega(x^0_N)$, then $Hx_N \leq h$ 
and
$(\|z_{k}\|_V, \beta_{k})\in \Omega(x^0_k)$ $\forall k>N$
if $\beta_k$ is equal to the rhs of~(\ref{eq:beta_term_bound}).
\end{lemma}

\begin{proof}
Suppose $(\|z_N\|_V,\beta_N)\in\Omega(x^0_N)$, then 
$Hx_N\leq h$ 
for all $x_N\in x^0_N+z_N+\mathcal{E}(V,\beta_N^2)$ 
since ${\|x^0_N + z_N\|_V + \beta_N \leq \hat{\rho}}$. If $\beta_{N+1}$ is equal to the rhs of (\ref{eq:beta_term_bound}), then from (\ref{eq:x0_term_bound}), (\ref{eq:z_term_bound}), (\ref{eq:Nc_def}) we have $(\|z_{N+1}\|_V,\beta_{N+1})\in\Omega(x^0_{N+1})$. 
By induction this argument implies $(\|z_{k}\|_V,\beta_{k})\in \Omega(x^0_k)$ for all $k>N$.
\end{proof}

\begin{remark}
If $\sigma^2/(1-\hat{\lambda}) < \hat{\rho}^2$, then
$\hat{N}$ in (\ref{eq:Nc_def}) is necessarily finite since
$\beta_{N+\hat{N}} = \sigma/(1-\hat{\lambda})^{\frac{1}{2}}$ is strictly feasible for~(\ref{eq:Nc_def}) in the limit as ${\hat{N}}\to\infty$.
The condition $(\|z_N\|_V,\beta_N)\in \Omega(x^0_N)$ introduces $2\hat{N}$ second order cone constraints and $\hat{N}$ scalar variables $\beta_{N+1},\ldots,\beta_{N+\hat{N}}$ into Problem~(\ref{opt:mpc}). 
The identity $(\hat{\lambda}\beta^2 + \sigma^2)^{\frac{1}{2}} \leq \hat{\lambda}^{\frac{1}{2}}\beta + \sigma$ implies a sufficient condition for~(\ref{eq:Nc_def}),
\begin{align}
\max_{(r,\beta_N)\in\Omega(x^0_N)} & \bigl\{
\hat{\lambda}^\frac{1}{2}\beta_{N+\hat{N}} + \sigma 
+ \hat{\lambda}^\frac{\hat{N}}{2}\! (r d_{\hat{\Phi}} +  d_\Theta L \|x^0_N\|_V) 
\nonumber \\
&\quad + \hat{\lambda}^\frac{\hat{N}+1}{2} (r +  \|x_k^0\|_V \bigr\} 
\leq \hat{\rho} ,
\label{eq:Nc_check}
\end{align}
which can be checked for given $\hat{N}$ by solving a SOCP. 
\end{remark}

The terminal cost $\hat{l}$ and parameter $\hat{\sigma}$ in Problem~(\ref{opt:mpc}) are constructed to ensure closed loop stability via the condition
\begin{multline}\label{eq:term_cost}
\hat{l}^2(z_{N+1},\beta_{N+1},x^0_{N+1}) - \hat{l}^2(z_N,\beta_N,x^0_N) 
\\
\leq -(\|x^0_N+z_N\|_{\hat{Q}} + \beta_N \|V^{-\frac{1}{2}}\|_{\hat{Q}})^2 + \hat{\sigma}^2 .
\end{multline}
Using variables $\beta_{N},\ldots,\beta_{N+\hat{N}}$ that appear in~(\ref{eq:term_constraint_set}), we define
\begin{align}
& 
\hat{l}^2(z_N,\beta_N,x^0_N) = \sum_{k=0}^{\hat{N}} 
l_{N+k}^2
\label{eq:hat_l_def}
\\
& l_{N+k} = \begin{cases}
\hat{\lambda}^{\frac{k}{2}}(\|x^0_N\|_V + \|z_N\|_V) + \beta_{N+k} & 0\leq k <\hat{N}
\\
\gamma \hat{\lambda}^{\frac{\hat{N}}{2}}(\|x^0_N\|_V + \|z_N\|_V) + \gamma\beta_{N+\hat{N}} &
k = \hat{N}
\end{cases}
\nonumber
\\
& \hat{\sigma} = 
\gamma \sigma + 
     \gamma  \hat{\lambda}^{\frac{\hat{N}}{2}} (d_{\hat{\Phi}}\hat{\rho} + d_\Theta L\|x^0_N\|_V)
\label{eq:hat_sigma_def}
\end{align}
where
$\gamma^2 = 1/(1-\hat{\lambda}^{\frac{1}{2}})$.

\begin{lemma}\label{lem:term_cost}
If $\hat{l}$ and $\hat{\sigma}$ are given by (\ref{eq:hat_l_def}) and (\ref{eq:hat_sigma_def}), then~(\ref{eq:term_cost}) holds for all $(z_N,\beta_N)$ satisfying $(\|z_N\|_V,\beta_N)\in\Omega(x^0_N)$.
\end{lemma}

\begin{proof}
This can be shown by inserting (\ref{eq:hat_l_def})-(\ref{eq:hat_sigma_def}) into (\ref{eq:term_cost}) and using (\ref{eq:x0_term_bound})-(\ref{eq:beta_term_bound}) with $\|z_N\|_V\leq \hat{\rho}$, which holds if $(\|z_N\|_V,\beta)\in\Omega(x^0_N)$, and by using the following inequality, which holds for any scalars $a,b,\hat{\lambda}>0$:
$(\hat{\lambda}^{\frac{1}{2}}a + b)^2 \leq 
\hat{\lambda}^{\frac{1}{2}} a^2 + b^2/(1-\hat{\lambda}^{\frac{1}{2}})$.
\end{proof}%

The procedure for computing the parameters defining the terminal cost and constraint set is summarised in Algorithm~\ref{alg:term}.%

{\setlength{\algomargin}{1em}
\removelatexerror
\begin{algorithm2e}
\DontPrintSemicolon
\SetKwInOut{Input}{Input}
\SetKwInOut{Output}{Output}
\Input{Bounds $\hat{\mathcal{X}}, \hat{\mathcal{U}}, \Theta_0$, matrices $\hat{A}^{(j)},\hat{B}^{(j)}$ in~(\ref{eq:ldi_term}); 
disturbance, state and control sets $\mathcal{W}$, $\mathcal{X}$, $\mathcal{U}$; 
scalar $\hat{\rho}$ in~(\ref{eq:rho_def}); 
cost weights $Q$, $R$}
\Output{$V$, $\sigma$, $K$, $\hat{\lambda}$, $\gamma$, $\hat{N}$, $\hat{\sigma}$}
{
At time $t=0$: Solve $(S^\star,Y^\star, \tau^\star) = \arg\min_{S,Y,\tau} \tau$ s.t.~(\ref{eq:cost_lmi}) and set $V\gets (S^\star)^{-1}$, $K\gets Y^\star V$, $\sigma^2 \gets \tau^\star$, 
$\hat{\lambda} \gets 1 - \sigma_{\min} ( V^{-\frac{1}{2}} \hat{Q} V^{-\frac{1}{2}} )$,
$\gamma \gets 2\|V^{-\frac{1}{2}}\|_{\hat{Q}}^2/(1-\hat{\lambda}^{\frac{1}{2}\!})$\;
At times $t\geq 0$ in steps 4 and 14 of Alg.~\ref{alg:mpc}: 
Set $\hat{N}$ to a prior estimate (e.g.~$\hat{N}\gets 1$) and check (\ref{eq:Nc_check}); increase $\hat{N}$ until~(\ref{eq:Nc_check}) is satisfied; 
compute $\hat{\sigma}$ using (\ref{eq:hat_sigma_def})
}
\caption{Computation of terminal parameters}\label{alg:term}
\end{algorithm2e}}

\begin{remark}\label{rem:terminal_update} 
The parameters $\hat{N},\hat{\sigma}$
depend on $x^0_N$ and are computed online in Alg.~\ref{alg:term}.
Although $V,K,\sigma,d_\Theta$ are defined in terms of $\Theta_0$ in Alg.~\ref{alg:term}, these may also be computed online using $\Theta_t$ without compromising the feasibility of Problem~(\ref{opt:mpc}). 
Moreover, time-varying $K_k$, $V_k$, $k\in\mathbb{N}_{[0,N]}$ may be employed using a straightforward extension of Lemma~\ref{lem:beta_mode1_dynamics} (see e.g.~\cite{Can11}).
\end{remark}

\begin{remark}\label{rem_const_basis_function}
It is often beneficial to include fixed vectors \( f_i \), \( i \in \mathcal{C} \), in the expansion~(\ref{eq:theta_expansion}), representing a constant disturbance to be estimated online. Define \( w_c = \sum_{i \in \mathcal{C}} \theta_i f_i \), then \( w_c \) lies in the convex hull of points \( w_c^{(q)} \), \( q \in \mathbb{N}_{\nu_\theta} \) determined by the vertices of the parameter set \( \Theta \).
To accommodate this, we replace \( w^{(r)} \) with \( w^{(r)} + w_c^{(q)} \) in~(\ref{eq:cost_lmi}) so that \( w_c \) is incorporated in the value of $\sigma$ and $f_i$ for any $i\not\in\mathcal{C}$ satisfies $f_i(0,0)=0$.
\end{remark}

\section{Recursive feasibility and stability}\label{sec:stability}

This section discusses the stability properties of the control law of Algorithm~\ref{alg:mpc}. 
Assuming Algorithm~\ref{alg:term} is feasible at time $t=0$ and Problem (\ref{opt:mpc}) is feasible for the initial nominal trajectory $\mathbf{x}^0$ and $\mathbf{v}^0$ at $t=0$, we show that the system~($\ref{eq:system}$) with the control law of Algorithm~\ref{alg:mpc} robustly satisfies constraints $(x_t,u_t)\in\mathcal{X}\times\mathcal{U}$ at all times $t \geq 0$. In addition, we provide an asymptotic performance bound and we show that the closed loop system is input-to-state practically stable (ISpS).

We first demonstrate that Problem~(\ref{opt:mpc}) in Alg.~\ref{alg:mpc} is feasible at each time step using an inductive argument considering the feasibility of (\ref{opt:mpc}) at iteration $i+1$ at time $t$ assuming  feasibility at iteration $i$ at time $t$, and the feasibility of (\ref{opt:mpc}) at the first iteration at time $t+1$ assuming feasibility at the final iteration of time $t$. In the latter case we use the following observation. 

\begin{lemma}\label{lem:beta_dynamics_term_bound}
If 
$x^0_k + \mathcal{S} \subseteq \hat{\mathcal{X}}$, $v^0_k=0$ and $Kx^0_k + K \mathcal{S} \subseteq \hat{\mathcal{U}}$, then 
\begin{multline}\label{eq:beta_dynamics_term_bound}
(\lambda_k \beta_k^2 + \sigma^2)^{\frac{1}{2}} + \max_{j\in\mathbb{N}_{\nu_1}, q\in\mathbb{N}_{\nu_\theta}} \lVert C^{(j)}_k z_k + \delta^{0\,(q)}_k \rVert_V\\
\leq
(\hat{\lambda}\beta_k^2 + \sigma^2)^{\frac{1}{2}} + d_{\hat{\Phi}} \|z_k\|_V + d_\Theta \| x^0_k \|_V
\end{multline}
for all $z_k\in\mathbb{R}^{n_x}$\!, $\beta_k\in\mathbb{R}$,  where $\lambda_k$, $\hat{\lambda}$ are given by (\ref{eq:lambda_mode1_def}), (\ref{eq:lambda_term}).
\end{lemma}

\begin{proof}
From (\ref{eq:lambda_term}) we have
$V - \frac{1}{\gamma}(Q+K^\top R K) \preceq \hat{\lambda}V$,
so~(\ref{eq:cost_lmi}) implies, for all $j\in \mathbb{N}_{\hat{\nu}}$ and $r \in \mathbb{N}_{n_w}$%
\def\strut{\rule{0pt}{8pt}}%
\begin{equation}\label{eq:lambda_mode1_cond_term}
\biggl[\begin{smallmatrix}
\hat{\lambda} V & ~~~0~~~ & \hat{\Phi}{\mbox{}^{(j)}}^\top
\\
\star\strut & \sigma^2 & {w^{(r)}}^\top
\\
\star\strut & \star & V^{-1}
\end{smallmatrix}\biggr] \succeq 0 .
\end{equation}
Furthermore, if $x^0_k + \mathcal{S} \subseteq \hat{\mathcal{X}}$, $v^0_k=0$ and ${Kx^0_k + K\mathcal{S}\subseteq \hat{\mathcal{U}}}$, then
$\mathrm{co}\{ \hat{\Phi}^{(j)}, \, j \in \mathbb{N}_{\hat{\nu}} \} \supseteq \mathrm{co}\{ \Phi_k + C_k^{(j)}, \, j\in\mathbb{N}_{\nu_1} \}$. Comparing  
(\ref{eq:lambda_mode1_cond_term}) and (\ref{eq:lambda_mode1_cond}) yields
$\hat{\lambda} \geq \lambda^{(j,r)}$ and 
$\hat{\lambda} \geq \lambda_k$. Since $d_{\hat{\Phi}} \|z_k\|_V + d_\Theta \|x^0_k\|_V \geq \max_{l, q} \lVert C^{(l)}_k z_k + \delta^{0\,(q)}_k \rVert_V$ we obtain (\ref{eq:lambda_mode1_cond_term}).
\end{proof}

\begin{theorem}\label{thm:feasibility}
If at time $t=0$, $\mathbf{v}^0$ and $x^0_0=x^p_0$ generate a nominal trajectory such that Problem~(\ref{opt:mpc}) is feasible, then 
at any iteration $i>1$ of Algorithm~\ref{alg:mpc} at $t\geq 0$, Problem~(\ref{opt:mpc}) is feasible with
$\mathbf{v}^0 = \mathbf{v}^0_{\mathrm{old}}$,
and at iteration $i=1$ at any time $t>0$, Problem~(\ref{opt:mpc}) is feasible if  $\mathbf{v}^0 = \mathbf{v}^0_{\mathrm{old}}$ and $x^0_0 = x^0_{0,\mathrm{old}}$.
\end{theorem}

\begin{proof}
This follows by induction from the following three cases. 
(i) If, at  iteration $\iter$ at time $t$, $\mathbf{v}^0$ and $x^0_0$ generate a nominal trajectory such that Problem~(\ref{opt:mpc}) in line~5 of Algorithm~\ref{alg:mpc} is feasible, then line~22 trivially ensures that~(\ref{opt:mpc}) is feasible at iteration $i+1$ time $t$ with $\mathbf{v}^0 = \mathbf{v}^0_{\mathrm{old}}$.
(ii) If, at the final iteration of time $t$, $\mathbf{v}^0$ and $x^0_0$ generate a nominal trajectory such that Problem~(\ref{opt:mpc}) in line~5 is feasible, then line~24 ensures 
(due to the definition of $\Omega$ and 
Lemmas~\ref{lem:term_tube}, \ref{lem:termset}, \ref{lem:term_cost}, and \ref{lem:beta_dynamics_term_bound})
that $\mathbf{v}^0=\mathbf{v}^0_{\mathrm{old}}$ and $x^0_0 = x^0_{0,\mathrm{old}}$ 
generate a nominal trajectory at iteration $i=1$ time $t+1$ such that Problem~(\ref{opt:mpc}) is feasible.
(iii) If, at any iteration $\iter$ and time $t$, Problem~(\ref{opt:mpc}) in line~5 is infeasible, then the feasibility of $\mathbf{v}^0 =\mathbf{v}^0_{\mathrm{old}}$ (or $\mathbf{v}^0 =\mathbf{v}^0_{\mathrm{old}}$ and $x^0_0=x^0_{0,\mathrm{old}}$ if $i=1$) implies that the line search in lines 8-21 necessarily terminates with $\mathbf{v}^0$ and $x^0_0$ that generate a nominal trajectory for which Problem~(\ref{opt:mpc}) is feasible, and hence feasibility of $\mathbf{v}^0 = \mathbf{v}^0_{\mathrm{old}}$ (or $\mathbf{v}^0 =\mathbf{v}^0_{\mathrm{old}}$ and $x^0_0=x^0_{0,\mathrm{old}}$) is ensured at iteration $i+1$ at time $t$ (or iteration $i=1$ at time $t+1$) by line~22 (or line~24, respectively).
\end{proof}

\begin{theorem}\label{thm:cost_bound}
If~(\ref{opt:mpc}) is feasible at $t=0$, then the control input and state of~(\ref{eq:system}) under Alg.~\ref{alg:mpc} satisfy $x_t\in \mathcal{X}$, $u_t\in\mathcal{U}$ and
\begin{equation}\label{eq:closedloop_cost}
\limsup_{T\rightarrow \infty} \frac{1}{T}\sum_{t=0}^{T-1} (\left\lVert x_t \right\rVert^2_Q +  \left\lVert u_t \right\rVert^2_R)\leq \bar{\sigma}^2 
\end{equation}
where $\bar{\sigma} = \gamma \sigma + 
\gamma \hat{\rho} (d_{\hat{\Phi}} + d_\Theta L)$.
\end{theorem}

\begin{proof}
From $(\|z_N\|_V,\beta_N)\in\Omega(x_N^0)$ we have $\|x_N^0\|_V \leq \hat{\rho}$ and hence   
$\hat{\sigma} = \gamma \sigma + \gamma\hat{\lambda}^{\frac{\hat{N}}{2}} (d_{\hat{\Phi}} \hat{\rho} + d_{\Theta}L\|x_N^0\|_V ) \leq \bar{\sigma}$.
The constraints of Problem~(\ref{opt:mpc}) therefore ensure, for all $t\geq 0$,
\begin{equation}\label{eq:cost_decrease}
\bar{J}^{(\mathit{final})}_{t+1} - \bar{J}^{(\mathit{final})}_{t} \leq -\bigl(\lVert x_{t} \rVert^2_Q +  \lVert u_{t} \rVert^2_R\bigr) + \bar{\sigma}^2 ,
\end{equation}
which implies~(\ref{eq:closedloop_cost}) because $\bar{J}^{(\mathit{final})}_{t}$ must be finite for all $t$.
\end{proof}

To complement the robust constraint satisfaction and performance bound of Theorem~\ref{thm:cost_bound}, closed loop stability can be characterized in terms of ISpS~\cite[Def.~6]{limon09} as follows.

\begin{theorem}\label{thm:ISpS_bound}
Let $d_{\mathcal{W}} = \max_{w\in\mathcal{W}} \|w\|$. If~(\ref{opt:mpc}) is feasible at $t=0$, then the state of~(\ref{eq:system}) under Alg.~1 satisfies, for all $t\geq 0$,
\begin{equation}\label{eq:ISpS_bound}
  \|x_t\| \leq \alpha_1(\|x_0\|, t) + \alpha_2(d_{\mathcal{W}}) + \alpha_3(d_\Theta) + \alpha_4(d_{\hat{\Phi}}),
\end{equation}
where $\alpha_1$ is a $\mathcal{K}\mathcal{L}$-function, $\alpha_2,\alpha_3,\alpha_4$ are $\mathcal{K}$-functions.
\end{theorem}

\begin{proof}
We first show that $\sigma$ is a $\mathcal{K}$-function of $d_{\mathcal{W}}$.
This follows from the definition of $\sigma$ as the minimal value of $\sqrt{\tau}$ satisfying (\ref{eq:cost_lmi}) and the observation that scaling the disturbance set $\mathcal{W}$ scales $\sigma$ by the same factor (namely, if $w^{(r)}\gets \kappa w^{(r)}$ for all $r$ and some $\kappa\in (0,1)$, then $\tau\gets \kappa^2 \tau$).
Moreover, $L$, $\gamma$ and $\hat{\rho}$ are bounded positive scalars ($L$ by assumption, $\gamma$ because $\hat{\lambda}\in[0,1)$, and $\hat{\rho}$ because $\hat{Q} \preceq V \preceq (1-\hat{\lambda})^{-1} \hat{Q}$).
%
Therefore, from $\bar{\sigma} = \gamma \sigma + \gamma \hat{\rho} (d_{\hat{\Phi}} + d_\Theta L)$, we have $\bar{\sigma} = \alpha_5(d_{\mathcal{W}}) + \alpha_6(d_\Theta) + \alpha_6(d_{\hat{\Phi}})$ where $\alpha_5,\alpha_6,\alpha_7$ are $\mathcal{K}$-functions.
Finally, the optimal objective of the SOCP~(\ref{opt:mpc}) is a continuous function of $x_t^p$, 
and (\ref{eq:cost_decrease}) with $x_t = x_t^p$ therefore implies $\bar{J}^{(\mathit{final})\!}_{t}$ is an ISpS-Lyapunov function~\cite[Def.~7]{limon09},
so~(\ref{eq:ISpS_bound}) follows from~\cite[Thm.~3]{limon09}.
\end{proof}

\section{Online learning of system parameters}\label{sec:param_est}

To use set membership estimation (SME)~\cite{Lor19,Lu21} to estimate $\theta$, we define  $D_t = D(x_t, u_t) = \bigl[f_1(x_t,u_t)\ \cdots \ f_{n_\theta}(x_t, u_t)\bigr]$, $d_t=d_t(x_t, u_t) = f_0(x_t,u_t)$, and 
rewrite (\ref{eq:system}) as:
\[
x_{t+1} = D_t \theta + d_t + w_t .
\]
Let $\Theta_t = \{\theta\in\mathbb{R}^{n_{\theta}} : H_\Theta \theta \leq h_t\}$ 
$\forall t\geq 0$, where $H_\Theta$ is fixed and $h_t$ is updated online using measured (or estimated) system states. This fixes the complexity (number of vertices) of the polytope $\Theta_t$. For example, if $H_\Theta = [-I_{n_\theta} \ \underline{1}]^\top$
then $\Theta_t$ is a simplex with $n_\theta+1$ vertices for all $t$.
Using an estimation horizon of length $N_{\Theta}\geq 1$, we determine $\Theta_{t}$ at time $t$ from the intersection of $\Theta_{t-1}$ with unfalsified parameter sets corresponding to $N_{\Theta}$ state observations. This requires the solution of a Linear Program (LP) for each row $i$ of $H_{\Theta}$:
\[
  [h_{t}]_i = \max_{\theta\in\Theta_{t-1}} [H_{\Theta}]_i \theta \ \text{subject to} \
\begin{aligned}[t]
x_{t+1-l} - D_{t-l} \theta - d_{t-l} \in \mathcal{W} &
  \\
\forall l \in \{ 1,\ldots, N_{\Theta}\} &
  \end{aligned}
\]
In \cite{Lu21} it is shown that the estimated parameter sets  satisfy $\Theta_{t} \subseteq \Theta_{t-1} \subseteq \cdots \subseteq \Theta_0$, and $\Theta_t$ converges to the true parameter vector $\theta^\ast$ under conditions ensuring persistency of excitation. 

\section{Numerical results} \label{sec:simex}

\def\strut{\rule{0pt}{8pt}}
\begin{table*}[ht]
\caption{Scaling of computation per iteration with problem dimensions}
\label{table:computation}
\centerline{\begin{tabular}{@{}l | r r r r r r r r r r @{}}
\hline
$(n_x,n_u,n_{\theta})$\strut & $(2,1,2)$ & $(4,2,2)$ & $(4,2,4)$ & $(6,2,4)$ & $(5,2,5)$ & $(6,2,6) $ & $(8,2,8) $ & $(8,4,8) $ & $(10,4,10) $ & $(12,4,12)$ 
\\ \hline
Variables\strut & 48 & 60 & 60 & 62 & 61 & 62 & 64 & 84 & 86 & 88
\\
Linear inequalities\strut & 57 & 97 & 97 & 117 & 107 & 117 & 137 & 177 & 197 & 217
\\
SOC constraints\strut & 294 & 474 & 1274 & 1774 & 2184 & 3454 & 7314 & 7314 & 13334 & 21994
\\
Execution time (\SI{}{\second})\strut & 0.051 & 0.143 & 0.434 & 0.562 & 0.683 & 1.09 & 2.53 & 4.23 & 5.69 & 11.08
\\
\hline
\end{tabular}}
\end{table*}

The proposed MPC law was tested using random examples with models~(\ref{eq:system})-(\ref{eq:theta_expansion}) containing quadratic nonlinearities:
\[
  f_0(x,u) = Ax + Bu, \ \ f_i(x,u) = \hat{e}_i [x]_{j_i}^2, \ i\in \mathbb{N}_{n_\theta}
\]
where $A,B$ are randomly chosen matrices,
$\hat{e}_i$ is the $i$th column of the identity matrix $I_{n_x}$,
$j_i$ is a randomly chosen index from $\mathbb{N}_{n_x}$,
and $A$, $B$ and $j_1,\ldots,j_{n_\theta}$ are known to the controller.
The disturbance set $\mathcal{W}$ belongs to a subspace of dimension $n_w$ and has vertices $w^{(r)} = B_w \hat{w}^{(r)}$, $r\in\mathbb{N}_{2^{n_w}}$, where $B_w\in\mathbb{R}^{n_x\times n_w}$ is a randomly generated full column-rank matrix and $B_w$, $\{\hat{w}^{(r)},\, r\in\mathbb{N}_{\nu_w}\}$ are known.
The true parameter $\theta^\ast$ and initial parameter set $\Theta_0$ are randomly chosen so that $\theta^\ast\in\Theta_0$ and $\max_{\theta\in\Theta_0}\| \theta - \theta^\ast\| \leq 0.05$.
For all $t$, $\Theta_t$ is a simplex updated using SME with estimation horizon $N_\Theta = 5$, and the nominal parameter $\theta^0_t$ is the mean of the vertices of $\Theta_t$.

The state and control sets are $\mathcal{X}= \mathbb{R}^{n_x}$, $\mathcal{U} = \{u: \| u \|_\infty \leq 1\}$, and the disturbance set is $\mathcal{W}=\{ B_w \hat{w} : \|\hat{w}\|_\infty \leq 0.01\}$. The cost matrices are $Q = I_{n_x}$, $R = I_{n_u}$, the prediction horizon is $N = 10$, and each simulation is run for $10$ time steps with a randomly chosen feasible initial condition.

The computation of $V,K$ (Algorithm~\ref{alg:term}, step 1) was performed using an LDI containing the system dynamics for all $x\in\hat{\mathcal{X}} = \{x : \|x\|_\infty \leq 1.5 \}$ and  $\theta\in\Theta_0$. The state perturbation constraint set is a simplex: $\mathcal{S} = \{s : [-I_{n_x} \ \underline{1}]^\top s \leq 0.5\underline{1}\}$. No control perturbation set is needed ($\mathcal{V} = \mathbb{R}^{n_u}$) because the model is linear in $u$. Since the uncertain terms in the model are quadratic, bounds on the errors $\delta^0$ and $\delta^1$ are determined by the vertices of $\Theta_t$ and $\mathcal{S}$. Hence in this case $\mathcal{W}^0$ and $\mathcal{W}^1$ have at most $\nu_\theta= n_\theta+1$ and $\nu_1 = (n_\theta+1)^2$ vertices.

We apply Algorithm \ref{alg:mpc} with solution $\algtol = 10^{-3}$ using MOSEK~\cite{mosek } with Yalmip~\cite{Yalmip} to solve Problem~(\ref{opt:mpc}) (M3 Pro, 36 GB memory).%
\footnote{Simulation code: {https://github.com/markcannon/ellipsoidal-anmpc}}
Table~\ref{table:computation} shows the average time for a single iteration of Algorithm~1, the main component of which is Problem~(\ref{opt:mpc}) in line 5, for various state, control and parameter dimensions $n_x$, $n_u$, $n_\theta$. In each case the disturbance dimension is $n_w=2$, and the computation time is the mean of $20$ randomly generated problems of a given size.
Computation is primarily determined
by the number of SOC constraints in~Problem~(\ref{opt:mpc}), which is dominated by the number of vertices of the set $\mathcal{W}^1$, and in this example scales as $O(n_xn_\theta^2N)$.

\begin{figure}[h]
\centerline{\includegraphics[scale=0.48, trim=15mm 78mm 22mm 84mm, clip]{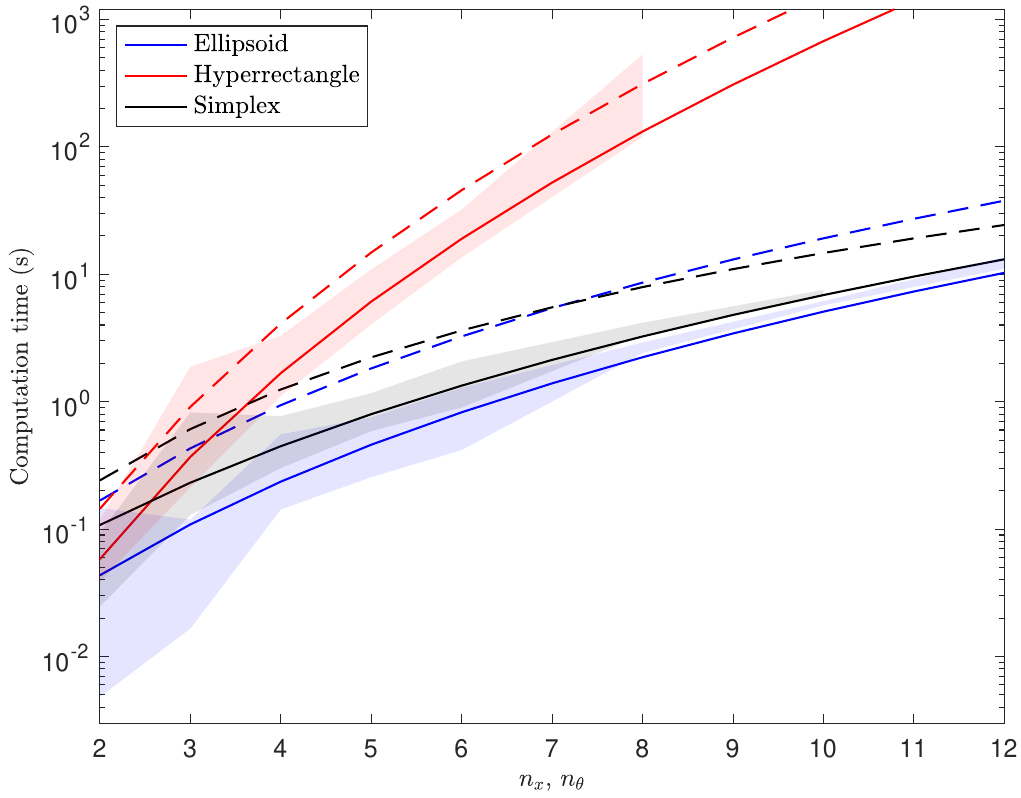}}
\caption{Average time required for one iteration (solid lines) and total time to convergence (dashed lines), 
as a function of $n_x=n_\theta$. Shaded regions indicate the upper and lower bounds on observed computation time per iteration.}%
\label{fig:timing}%
\end{figure}

\begin{figure}[h]
\centerline{\includegraphics[scale=0.48, trim=15mm 78mm 22mm 84mm, clip]{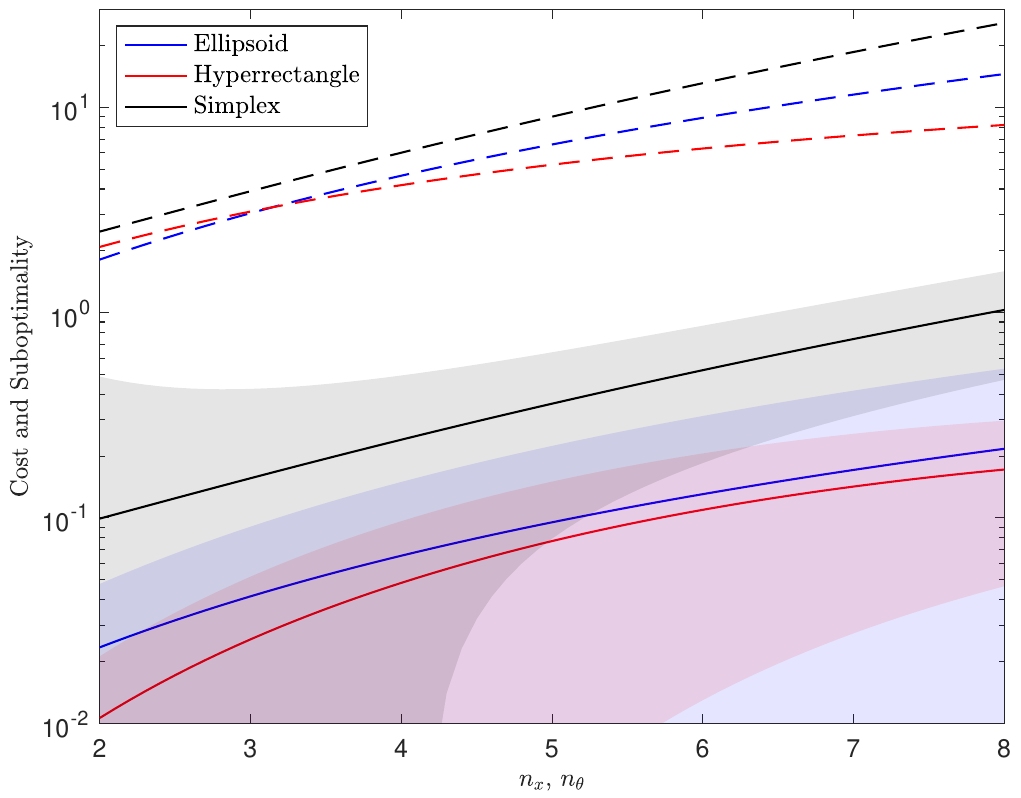}}
\caption{Suboptimality of the first iteration at $t=0$, as a fraction of the optimal cost for no uncertainty: meam values (solid lines) and 2 standard deviation bounds (shaded regions); and mean cost evaluated for closed loop trajectories over $10$ time steps (dashed lines),
as a function of $n_x=n_\theta$.}%
\label{fig:suboptimality}%
\end{figure}

As discussed in Remark~\ref{rem:terminal_update}, the algorithm parameters may be updated online using the current estimate $\Theta_t$ by solving a SDP to minimize $\tau$ subject to (\ref{eq:cost_lmi}). This requires around $50$\,ms for $n_x= n_\theta= 2$, which is approximately the same as the average time for one iteration of the corresponding online MPC optimization. However, the SDP computation grows faster than that of Problem~(\ref{opt:mpc}),
and for $n_x=n_\theta= 8$ requires more than $100$\,s on average, or 50 times the computation for one iteration of the online MPC optimization. Therefore the decision to recompute algorithm parameters online will depend on the problem size and available computational resources.

To provide a comparison with robust polytopic tube NMPC, we implement the controller of~\cite{Bue24} on the same randomly generated benchmark problem set. We consider configuration-constrained polytopic tubes~\cite{villanueva24} with  polytopic cross-sections parameterized as $\mathcal{X}_k = \{x: H x \leq h_k\}$, where $H$ is fixed and $h_k$,  $k\in\mathbb{N}_{[0,N]}$ are decision variables optimized online. We evaluate two choices for $H$: hyperrectangular tubes with $H = [I_{n_x} \  {-I}_{n_x}]^{\top\!\!}$ and simplex tubes with $H = [{-I}_{n_x} \ \underline{1}]^{\top\!\!}$. The number of vertices depends exponentially on $n_x$ for hyperrectangular tubes, and linearly for simplex tubes.
In both cases the online NMPC problem is solved via successive convexification using a difference of convex (DC) representation of the nonlinear dynamics, as in~\cite{Bue24}. The quadratic nonlinearity in this example allows the DC model to be obtained in closed form, and the resulting MPC optimization can be solved as a SOCP, enabling a fair comparison of computation and optimality using the same solver (MOSEK with Yalmip).

The numbers of SOC constraints required by hyperrectangular tube MPC and by simplex tube MPC scale as 
$O(2^{n_x}n_\theta^2N)$ 
and
$O(n_xn_\theta^2N)$,
respectively.
Consequently the computation per iteration for hyperrectangular tubes depends exponentially on $n_x$ and rapidly exceeds that of the ellipsoidal tube algorithm (Fig.\,\ref{fig:timing}), whereas the computation times for simplex tube MPC and ellipsoidal tube MPC are similar for all $n_x$, $n_\theta$.

Ellipsoidal tube MPC needs more iterations than polytopic tube MPC at each time step (2--2.5 on average compared to 3.5--4.5 across the range of problem sizes) to reach the same convergence conditions, causing the total computation per time step for ellipsoidal tube MPC to exceed both polytopic tube MPC variants for small $n_x$ and $n_\theta$ (Fig.\,\ref{fig:timing}).
This is due to the DC decomposition used in~\cite{Bue24}, which allows more accurate convex problem approximations without needing the state and control perturbation bounds $\mathcal{S}$ and $\mathcal{V}$ that were introduced in Section~\ref{sec:error_bounds} to bound linearization errors. 
However, for larger problems ($n_x$ and $n_\theta$ greater than $3$), ellipsoidal tube MPC requires significantly less total computation time than hyperrectangular tube MPC  (Fig.\,\ref{fig:timing}), while the degree of suboptimality of ellipsoidal and hyperrectangular tube MPC remain comparable as the problem size increases (Fig.\,\ref{fig:suboptimality}). Simplex tubes result in greater suboptimality since the bounds they provide on uncertain system trajectories are more conservative for  this set of problems. 
Thus, as the number of states and  unknown parameters increase, the proposed ellipsoidal tube approach provides a favourable compromise between computation and performance than either polytopic tube scheme.

\section{Conclusion}\label{sec:conclusion}
This note introduces a robust nonlinear MPC strategy with online parameter adaptation based on ellipsoidal tubes. The algorithm is recursively feasible, ensures robust closed loop stability, and requires a sequence of convex optimization to be solved online. The computational requirement scales favourably with the dimensions of the model and uncertain parameters due to the use of ellipsoidal tubes.
A promising research direction is to develop a bespoke first order solver for the online optimization to more effectively exploit warm-starting and to allow parallelization of computation. 
Another potential extension is to consider convexification methods using differences of convex functions (as is done in~\cite{Bue24} with polytopic tubes) in the context of ellipsoidal tubes.

\bibliographystyle{IEEEtran}
\bibliography{ellipsoidal}

\end{document}